\documentclass[reqno]{amsart}

\usepackage{amsaddr}
\usepackage{amssymb,amsmath,amsthm,xcolor,enumerate,hyperref,cleveref,mathrsfs}

\allowdisplaybreaks

\newtheorem{thrm}{Theorem}[section]
\newtheorem{cor}[thrm]{Corollary}
\newtheorem{lem}[thrm]{Lemma}

\theoremstyle{definition}

\crefrangeformat{equation}{#3(#1)#4--#5(#2)#6}

\crefname{thrm}{Theorem}{Theorems}
\crefname{lem}{Lemma}{Lemmas}
\crefname{cor}{Corollary}{Corollaries}
\crefname{prop}{Proposition}{Propositions}
\crefname{defn}{Definition}{Definitions}
\crefname{exm}{Example}{Examples}
\crefname{rem}{Remark}{Remarks}
\crefname{section}{Section}{Sections}
\crefname{equation}{\unskip}{\unskip}
\crefname{enumi}{\unskip}{\unskip}

\newcommand{\ad}{\mathrm{ad}}

\begin{document}

\title[Local Derivations of Finitary Incidence Algebras]{Local Derivations of Finitary Incidence Algebras}

\author{Mykola Khrypchenko}
\address{Departamento de Matem\'atica, Universidade Federal de Santa Catarina,  Campus Reitor Jo\~ao David Ferreira Lima, Florian\'opolis, SC, CEP: 88040--900, Brazil}
\email{nskhripchenko@gmail.com}

\begin{abstract}
Let $P$ be a partially ordered set, $R$ a commutative ring with identity and $FI(P,R)$ the finitary incidence algebra of $P$ over $R$. In this note we prove that each $R$-linear local derivation of $FI(P,R)$ is a derivation, which partially generalizes Theorem 3 of \cite{Nowicki-Nowosad04}.
\end{abstract}



\maketitle

\section*{Introduction}\label{intro}

Local derivations appeared in the early 90's in the works by Kadison~\cite{Kadison90} and Larson-Sourour~\cite{Larson-Sourour}. Kadison proved in~\cite[Theorem A]{Kadison90} that each local derivation of a von Neumann algebra with values in its dual bimodule is a derivation. Bre\v{s}ar showed in~\cite{Bresar92} that Theorem A by Kadison remains valid for any normed bimodule. The main result of Larson and Sourour~\cite{Larson-Sourour} says that the algebra of all bounded operators on a complex infinite-dimensional Banach space has no proper local derivations. An alternative proof of this fact (which also works in the real case) was given in~\cite{Bresar-Semrl93}. In the case of $2$-local derivations one can even drop the linearity and continuity as was shown by \v{S}emrl in~\cite{Semrl97}.

The incidence algebra $I(P,R)$ of a locally finite preordered set $P$ over a commutative ring $R$ is a classical object in the area of derivations and their generalizations. When $|P|=n<\infty$, the algebra $I(P,R)$ can be seen as a subalgebra of the full matrix algebra $M_n(R)$, and by this reason $I(P,R)$ is sometimes called a structural matrix algebra. We would like to note that $M_n(R)$, as well as its subalgebra $T_n(R)$ of upper triangular matrices over $R$, are particular cases of $I(P,R)$. On the other hand, if $P$ is finite and connected with $|P|\ge 2$, then $I(P,R)$ is a triangular algebra~\cite{Zhang-Yu06}\footnote{In~\cite{Anderson-Winders09} such an algebra is called the idealization of a bimodule.} (when $P$ is finite, but not necessarily connected, one has $I(P,R)=\bigoplus_{j=1}^k I(P_j,R)$, where $P_1,\dots,P_k$ are the connected components of $P$, so if each $P_j$ has at least $2$ elements, then $I(P,R)$ is a direct sum of triangular algebras). The case of finite $P$ is easier to deal with, 
since $I(P,R)$ possesses the natural basis formed by matrix units, and it only suffices to study the behavior of a derivation on the elements of the basis (see~\cite{Nowicki83,Nowicki84-der,Coelho-Polcino93,Jondrup95,Nowicki-Nowosad04,Benkovic05,Benkovic07,Ghosseiri07,Chen-Zhang08,Zhao-Yao-Wang10,Alizadeh-Bitafaran}). In the infinite case the latter does not work (unless one imposes some extra restrictions as in~\cite{Xiao15}), and some other technique is needed (see~\cite{Baclawski72,Scharlau74,Koppinen95,Khripchenko12,Khrypchenko16,Zhang-Khrypchenko}).

Based on an earlier work by Nowicki~\cite{Nowicki83}, Nowicki and Nowosad proved in~\cite[Theorem 3]{Nowicki-Nowosad04} that each $R$-linear local derivation of $I(P,R)$ is a derivation, provided that $P$ is a finite preordered set and $R$ is a commutative ring. Alizadeh and Bitarafan improved a particular case of~\cite[Theorem 3]{Nowicki-Nowosad04} by showing in~\cite[Theorem 3.7]{Alizadeh-Bitafaran} that $M_n(R)$ has no proper (additive, but not necessarily $R$-linear) local derivations with values in $M_n(\mathcal{M})$, where $\mathcal{M}$ is $2$-torsion free central $R$-bimodule and $n\ge 3$. Applying arguments similar to those used by Nowicki and Nowosad~\cite{Nowicki-Nowosad04}, Zhao, Yao and Wang proved in~\cite[Theorem 2.1]{Zhao-Yao-Wang10} that each local Jordan derivation of $T_n(R)$ is a derivation.

In this short note, which was inspired by the recent preprint~\cite{CDH} by Courtemanche, Dugas and Herden, we adapt the ideas from~\cite{Nowicki-Nowosad04} to the infinite case using the technique elaborated in~\cite{Khripchenko12,Khrypchenko16}. More precisely, we show that each $R$-linear local derivation of the finitary incidence algebra $FI(P,R)$ of an arbitrary poset $P$ over a commutative unital ring $R$ is a derivation, giving thus another partial generalization of~\cite[Theorem 3]{Nowicki-Nowosad04}.

\section{Preliminaries}\label{prelim}

Let $R$ be a ring. An additive map $d:R\to R$ is called a {\it derivation} of $R$, if it satisfies
\begin{align*}
 d(rs)=d(r)s+rd(s)
\end{align*}
for all $r,s\in R$. Each $a\in R$ defines the derivation $\ad_a$, given by $\ad_a(r)=ar-ra$. A derivation of such a form is called {\it inner}. A {\it local derivation}~\cite{Kadison90,Larson-Sourour} of $R$ is an additive map $d:R\to R$, such that for any $r\in R$ there is a derivation $d_r$ of $R$ with $d(r)=d_r(r)$. Obviously, each derivation of $R$ is a local derivation of $R$. Observe also that for any local derivation $d$ of $R$ and any idempotent $e$ of $R$ one has
\begin{align}\label{d(e)=d(e)e+ed(e)}
 d(e)=d_e(e)=d_e(e)e+ed_e(e)=d(e)e+ed(e).
\end{align}

Let $(P,\le)$ be a partially ordered set and $R$ a commutative ring with identity. With any pair of $x\le y$ from $P$ associate a symbol $e_{xy}$ and denote by $I(P,R)$ the $R$-module of formal sums
\begin{align}\label{formal-sum}
 \alpha=\sum_{x\le y}\alpha(x,y)e_{xy},
\end{align}
where $\alpha(x,y)\in R$. If $x$ and $y$ run through a subset $X$ of the ordered pairs $x\le y$ in the sum \cref{formal-sum}, then it is meant that $\alpha(x,y)=0$ for any pair $x\le y$ which does not belong to $X$.

The sum \cref{formal-sum} is called a {\it finitary series}~\cite{KN-fi}, whenever for any pair of $x,y\in P$ with $x<y$ there exists only a finite number of $u,v\in P$, such that $x\le u<v\le y$ and $\alpha(u,v)\ne 0$. The set of finitary series, denoted by $FI(P,R)$, is an $R$-submodule of $I(P,R)$ which is closed under the convolution of the series:
\begin{align}\label{conv-series}
 \alpha\beta=\sum_{x\le y}\left(\sum_{x\le z\le y}\alpha(x,z)\beta(z,y)\right)e_{xy}
\end{align}
for $\alpha,\beta\in FI(P,R)$. Thus, $FI(P,R)$ is an $R$-algebra, called the {\it finitary incidence algebra of $P$ over $R$}. Moreover, $I(P,R)$ is a bimodule over $FI(P,R)$ under \cref{conv-series}.

\section{Local derivations of $FI(P,R)$}

Given $x\le y$, we identify $e_{xy}$ with the series $1_Re_{xy}\in FI(P,R)$. Note that
\begin{align}\label{e_xy-times-e_uv}
 e_{xy}e_{uv}=\delta_{yu}e_{xv},
\end{align}
where $\delta$ is the Kronecker delta. In particular, the elements $e_x:=e_{xx}$ are orthogonal idempotents of $FI(P,R)$, and for any $\alpha\in FI(P,R)$ one has
\begin{align}\label{e_x-alpha-e_y}
 e_x\alpha e_y=
 \begin{cases}
  \alpha(x,y)e_{xy}, & x\le y,\\
  0, & x\not\le y.
 \end{cases}
\end{align}
We shall also consider the idempotents $e_X:=\sum_{x\in X}1_Re_{xx}\in FI(P,R)$, where $X\subseteq P$.

For any $\alpha\in FI(P,R)$ and $x\le y$ we define
\begin{align}\label{alpha|_x^y}
 \alpha|_x^y=\alpha(x,y)e_{xy}+\sum_{x\le v<y}\alpha(x,v)e_{xv}+\sum_{x<u\le y}\alpha(u,y)e_{uy}.
\end{align}
Observe that the sums in \cref{alpha|_x^y} are finite, so $\alpha\mapsto\alpha|_x^y$ is a well-defined map $FI(P,R)\to FI(P,R)$. Moreover, it is $R$-linear and satisfies
\begin{align}
(\alpha|_x^y)|_x^y&=\alpha|_x^y,\label{double-restr=restr}\\
 (e_X)|_x^y&=e_{X\cap\{x,y\}}.\label{restr-of-e_X}
\end{align}
The next result is a partial generalization of \cite[Lemma 8]{Khripchenko12}.

\begin{lem}\label{d(alpha)(xy)=d(alpha_x^y)(xy)}
 For each $R$-linear local derivation $d$ of $FI(P,R)$ and $x\le y$ one has
\begin{align}\label{d(alpha)(xy)}
 d(\alpha)(x,y)=d(\alpha|_x^y)(x,y).
\end{align} 
\end{lem}
\begin{proof}
 We first assume that $d$ is an $R$-linear derivation of $FI(P,R)$. By \cref{e_x-alpha-e_y}
 $$
 d(\alpha(x,y)e_{xy})=d(e_x\alpha e_y)=d(e_x)\alpha e_y+e_xd(\alpha)e_y+e_x\alpha d(e_y),
 $$
 whence
 \begin{align}\label{d(af)(x_y)=d(af|_x^y)(x_y)}
  d(\alpha)(x,y)&=d(\alpha(x,y)e_{xy})(x,y)-(d(e_x)\alpha)(x,y)-(\alpha d(e_y))(x,y).
 \end{align}
In view of \cref{conv-series,alpha|_x^y} the right-hand side of \cref{d(af)(x_y)=d(af|_x^y)(x_y)} is
\begin{align*}
d((\alpha|_x^y)(x,y)e_{xy})(x,y)-(d(e_x)\alpha|_x^y)(x,y)-(\alpha|_x^y d(e_y))(x,y),
\end{align*}
which is $d(\alpha|_x^y)(x,y)$ by the same \cref{d(af)(x_y)=d(af|_x^y)(x_y)}, whence \cref{d(alpha)(xy)}.

Now let $d$ be an $R$-linear local derivation of $FI(P,R)$. Then using the result of the previous case and \cref{double-restr=restr}
\begin{align*}
 d(\alpha)(x,y)&=d(\alpha-\alpha|_x^y)(x,y)+d(\alpha|_x^y)(x,y)\\
 &=d_{\alpha-\alpha|_x^y}((\alpha-\alpha|_x^y)|_x^y)(x,y)+d(\alpha|_x^y)(x,y)\\
 &=d_{\alpha-\alpha|_x^y}(\alpha|_x^y-(\alpha|_x^y)|_x^y)(x,y)+d(\alpha|_x^y)(x,y)\\
 &=d(\alpha|_x^y)(x,y),
\end{align*}
which proves \cref{d(alpha)(xy)}.
\end{proof}

We shall also need the following lemma which partially generalizes Lemma 1 from~\cite{Khripchenko12}.
\begin{lem}\label{d(e_X)(u_v)=d(e_u)(u_v)-or-d(e_v)(u_v)}
 Let $d$ be an $R$-linear local derivation of $FI(P,R)$ and $X\subseteq P$. Then for all $u\le v$ one has 
 \begin{align}\label{d(e_X)-in-terms-of-d(e_u)}
  d(e_X)(u,v)=\begin{cases}
          d(e_u)(u,v), & \mbox{if $u\in X$ and $v\not\in X$},\\
          d(e_v)(u,v), & \mbox{if $u\not\in X$ and $v\in X$},\\
          0, & \mbox{otherwise}.
         \end{cases}
\end{align}
\end{lem}
\begin{proof}
 The first two cases of \cref{d(e_X)-in-terms-of-d(e_u)}, as well as the case $u,v\not\in X$, are immediate consequences of \cref{restr-of-e_X,d(alpha)(xy)=d(alpha_x^y)(xy)}. Now let $u,v\in X$. Then $d(e_X)(u,v)=d_{e_X}(e_X)(u,v)$, the latter being zero by \cite[Lemma 1]{Khripchenko12}.
\end{proof}

\begin{cor}\label{d(e_x)(xy)-and-d(e_y)(xy)}
 Let $d$ be an $R$-linear local derivation of $FI(P,R)$ and $x\le y$. Then
 \begin{align}\label{d(e_x)(xy)=-d(e_y)(xy)}
 d(e_x)(x,y)=-d(e_y)(x,y).
 \end{align} 
\end{cor}
\noindent Indeed, if $x=y$, then $d(e_x)(x,x)=0$ thanks to \cref{d(e_X)(u_v)=d(e_u)(u_v)-or-d(e_v)(u_v)}, and if $x<y$, then $d(e_x+e_y)(x,y)=d(e_{\{x,y\}})(x,y)=0$ by the same reason.\\

The following fact is a partial generalization of~\cite[Lemma 2]{Khripchenko12}.
\begin{lem}\label{d(e_x)=ad_alpha(e_x)}
 Let $d$ be an $R$-linear local derivation of $FI(P,R)$. Then there exists $\alpha\in FI(P,R)$ such that $d(e_x)=\ad_\alpha(e_x)$ for all $x\in P$.
\end{lem}
\begin{proof}
 Define 
 \begin{align*}
 \alpha=\sum_{x\le y}d(e_y)(x,y)e_{xy}\in I(P,R).
 \end{align*}
 Then $\alpha e_x=d(e_x)e_x$, and since by \cref{d(e_x)(xy)=-d(e_y)(xy)}
  \begin{align*}
 \alpha=-\sum_{x\le y}d(e_x)(x,y)e_{xy},
 \end{align*} 
one similarly has $e_x\alpha=-e_xd(e_x)$. So, by \cref{d(e)=d(e)e+ed(e)}
 $$
 d(e_x)=d(e_x)e_x+e_xd(e_x)=\alpha e_x-e_x\alpha=\ad_\alpha(e_x).
 $$
 It remains to prove that $\alpha\in FI(P,R)$. Suppose that there is an infinite set $S$ of pairs $(x_i,y_i)$, such that $x\le x_i<y_i\le y$ and $\alpha(x_i,y_i)\ne 0$. For each fixed $u$ there is only a finite number of $i$ such that $x_i=u$, as $d(e_u)(u,y_i)=-\alpha(x_i,y_i)\ne 0$ for such $u$ and $d(e_u)$ is a finitary series. Similarly for each $v$ there is only a finite number of $j$ such that $y_j=v$. Using this observation, similarly to what was done in the proof of~\cite[Lemma 2]{Khripchenko12}, we may construct an infinite $S'\subseteq S$, such that for any two pairs $(x_i,y_i)$ and $(x_j,y_j)$ from $S'$ one has $x_i\ne y_j$. Let $X=\{x_i\mid (x_i,y_i)\in S'\}$. Note that $y_i\not\in X$ for any $(x_i,y_i)\in S'$. So, using \cref{d(e_X)(u_v)=d(e_u)(u_v)-or-d(e_v)(u_v)}, we have for all $(x_i,y_i)\in S'$
 \begin{align*}
  d(e_X)(x_i,y_i)&=d(e_{X\setminus\{x_i\}}+e_{x_i})(x_i,y_i)\\
  &=d(e_{X\setminus\{x_i\}})(x_i,y_i)+d(e_{x_i})(x_i,y_i)\\
  &=d(e_{x_i})(x_i,y_i)=-\alpha(x_i,y_i)\ne 0.
 \end{align*}
 This contradicts the fact that $d(e_X)\in FI(P,R)$.
\end{proof}

It follows from \cref{d(e_x)=ad_alpha(e_x)} that it suffices to describe the local derivations of $FI(P,R)$ which satisfy 
\begin{align}\label{d(e_x)-is-zero}
 d(e_x)=0
\end{align}
for all $x\in P$.

\begin{lem}\label{d(e_x)=0}
 Let $d$ be an $R$-linear local derivation of $FI(P,R)$ satisfying \cref{d(e_x)-is-zero} for all $x\in P$. Then there exists $\sigma\in I(P,R)$, such that 
 \begin{align}\label{d(alpha)(xy)=sigma(xy)alpha(xy)}
 d(\alpha)(x,y)=\sigma(x,y)\alpha(x,y)  
 \end{align}
for all $\alpha\in FI(P,R)$ and $x\le y$.
\end{lem}
\begin{proof}
 We first show that 
 \begin{align}\label{d(e_xy)(uv)=0-for-[uv]ne[xy]}
  d(e_{xy})(u,v)=0\mbox{ for }(u,v)\ne(x,y).
 \end{align}
 In view of \cref{d(e_x)-is-zero}, equality \cref{d(e_xy)(uv)=0-for-[uv]ne[xy]} is trivial, when $x=y$. For $x<y$ observe by \cref{d(alpha)(xy)=d(alpha_x^y)(xy)} that
 \begin{align}\label{d(e_xy)_uv}
 d(e_{xy})(u,v)=d((e_{xy})|_u^v)(u,v).
 \end{align}
The latter may be non-zero in the following two cases: 
\begin{enumerate}
 \item $u=x<y\le v$;\label{u=x<y-le-v}
 \item $u\le x<y=v$.\label{u-le-x<y=v}
\end{enumerate}

 \cref{u=x<y-le-v} Let $u=x<y<v$. Notice from \cref{e_xy-times-e_uv} that $e_y+e_{xy}$ is an idempotent of $FI(P,R)$, so by \cref{d(e)=d(e)e+ed(e),d(e_x)-is-zero,d(e_xy)_uv}
 \begin{align*}
  d(e_y+e_{xy})(u,v)&=(d(e_y+e_{xy})(e_y+e_{xy})+(e_y+e_{xy})d(e_y+e_{xy}))(x,v)\\
  &=d(e_y+e_{xy})(y,v)=d(e_{xy})(y,v)=0.
 \end{align*}

 \cref{u-le-x<y=v} Let $u<x<y=v$. Considering the idempotent $e_x+e_{xy}\in FI(P,R)$, as above we get
 \begin{align*}
  d(e_x+e_{xy})(u,v)&=(d(e_x+e_{xy})(e_x+e_{xy})+(e_x+e_{xy})d(e_x+e_{xy}))(u,y)\\
  &=d(e_x+e_{xy})(u,x)=d(e_{xy})(u,x)=0,
 \end{align*}
 completing the proof of \cref{d(e_xy)(uv)=0-for-[uv]ne[xy]}.
 
 Define 
 \begin{align}\label{sigma=sum-d(e_xy)(xy)}
 \sigma=\sum_{x\le y}d(e_{xy})(x,y)e_{xy}\in I(P,R).
 \end{align}
 Using \cref{d(alpha)(xy)=d(alpha_x^y)(xy),alpha|_x^y,d(e_xy)(uv)=0-for-[uv]ne[xy]} and linearity of $d$ we conclude that
 \begin{align*}
  d(\alpha)(x,y)=d(\alpha|_x^y)(x,y)=\alpha(x,y)d(e_{xy})(x,y)=\sigma(x,y)\alpha(x,y).
 \end{align*}

\end{proof}

\begin{lem}\label{sigma-cocycle}
 Let $d$ be as in \cref{d(e_x)=0}. Then the corresponding element $\sigma\in I(P,R)$ given by \cref{sigma=sum-d(e_xy)(xy)} satisfies 
 \begin{align}\label{sigma(xy)+sigma(yz)=sigma(xz)}
 \sigma(x,y)+\sigma(y,z)=\sigma(x,z)  
 \end{align}
for all $x\le y\le z$.
\end{lem}
\begin{proof}
 Clearly, \cref{sigma(xy)+sigma(yz)=sigma(xz)} holds, when $x=y$ or $y=z$, thanks to \cref{d(e_x)-is-zero,sigma=sum-d(e_xy)(xy)}. Suppose that $x<y<z$ and take 
 \begin{align}\label{alpha=e_xy+e_yz-e_xz-e_y}
 \alpha=e_{xy}+e_{yz}-e_{xz}-e_y.  
 \end{align}
Then by \cref{d(alpha)(xy)=sigma(xy)alpha(xy),alpha=e_xy+e_yz-e_xz-e_y,d(alpha)(xy)=d(alpha_x^y)(xy),alpha|_x^y} we have
 \begin{align*}
  \sigma(x,y)&=d(\alpha)(x,y)=d_{\alpha}(\alpha)(x,y)=d_{\alpha}(\alpha|_x^y)(x,y)=d_{\alpha}(e_{xy}-e_y)(x,y),\\
  \sigma(y,z)&=d(\alpha)(y,z)=d_{\alpha}(\alpha)(y,z)=d_{\alpha}(\alpha|_y^z)(y,z)=d_{\alpha}(e_{yz}-e_y)(y,z),\\
  -\sigma(x,z)&=d(\alpha)(x,z)=d_{\alpha}(\alpha)(x,z)=d_{\alpha}(\alpha|_x^z)(x,z)=d_{\alpha}(e_{xy}+e_{yz}-e_{xz})(x,z).
 \end{align*}
Adding these equalities, we get
\begin{align}
 \sigma(x,y)+\sigma(y,z)-\sigma(x,z)&=d_{\alpha}(e_{xy})(x,y)+d_{\alpha}(e_{yz})(y,z)-d_{\alpha}(e_{xz})(x,z)\label{d(e_xy)(xy)+d(e_yz)(yz)-d(e_xz)(xz)}\\
 &\quad-d_{\alpha}(e_y)(x,y)+d_{\alpha}(e_{yz})(x,z)\label{-d(e_y)(xy)+d(e_yz)(xz)}\\
 &\quad-d_{\alpha}(e_y)(y,z)+d_{\alpha}(e_{xy})(x,z).\label{-d(e_y)(yz)+d(e_xy)(xz)}
\end{align}
Observe that the right-hand side of \cref{d(e_xy)(xy)+d(e_yz)(yz)-d(e_xz)(xz)} is zero by \cite[Lemma 4]{Khrypchenko16}. To show that  \cref{-d(e_y)(xy)+d(e_yz)(xz),-d(e_y)(yz)+d(e_xy)(xz)} are also zero, write
\begin{align*}
 d_\alpha(e_{yz})(x,z)&=d_\alpha(e_ye_{yz})(x,z)=(d_\alpha(e_y)e_{yz}+e_yd_\alpha(e_{yz}))(x,z)=d_\alpha(e_y)(x,y),\\
 d_\alpha(e_{xy})(x,z)&=d_\alpha(e_{xy}e_y)(x,z)=(d_\alpha(e_{xy})e_y+e_{xy}d_\alpha(e_y))(x,z)=d_\alpha(e_y)(y,z).
\end{align*}
\end{proof}

\begin{thrm}\label{loc-der-of-FI-is-der,}
 Each $R$-linear local derivation of $FI(P,R)$ is a derivation.
\end{thrm}
\begin{proof}
 By \cref{d(e_x)=ad_alpha(e_x),d(e_x)=0,sigma-cocycle} each $R$-linear local derivation of $FI(P,R)$ is a sum of an inner derivation and a map of the form \cref{d(alpha)(xy)=sigma(xy)alpha(xy)} with $\sigma$ satisfying \cref{sigma(xy)+sigma(yz)=sigma(xz)}. It is readily checked by a direct application of \cref{conv-series} that such a map \cref{d(alpha)(xy)=sigma(xy)alpha(xy)} is a derivation (see also \cite[Lemma 3]{Khripchenko12} for a similar construction).
\end{proof}


\bibliography{bibl}{}
\bibliographystyle{acm}

\end{document}